\documentclass{article}

\usepackage{amssymb}
\usepackage{amsmath}
\usepackage{amsthm}
\usepackage{pictex}
\usepackage{graphicx}
\usepackage{verbatim}
\usepackage{hyperref, graphics}

\begin{document}

\theoremstyle{plain} \newtheorem{theo}{Theorem}
\theoremstyle{plain} \newtheorem{prop}{Proposition}
\theoremstyle{plain}	\newtheorem{lemma}[theo]{Lemma}
\theoremstyle{plain}	\newtheorem{cor}{Corollary}
\theoremstyle{definition}	\newtheorem{defi}{Definition}
\theoremstyle{corollary}	\newtheorem{corollary}{Corollary}
\theoremstyle{remark} \newtheorem{example}{Example}
\newcommand{\ud}{\,\mathrm{d}}

\newcommand{\mathsym}[1]{{}}
\newcommand{\unicode}[1]{{}}
\title{Spectral gap for quantum graphs \\
and their connectivity}

\author{P.\,Kurasov, G.\,Malenov\'a, and S.\,Naboko
\thanks{PK was supported in part by Swedish Research Council Grant N50092501 SN would like to
thank Stockholm university for support and hospitalityˆ Support by the Institut Mittag-Le?er Djursholm, Sweden is gratefully acknowledged.}}
\maketitle

\begin{abstract}
The spectral gap 
for Laplace operators on metric graphs is investigated in relation to graph's connectivity,
in particular what happens if an edge is added to (or deleted from) a graph. It is shown that in contrast
to discrete graphs connection between the connectivity and the spectral gap is not one-to-one.
The size of the spectral gap depends not only on the topology of the metric graph but on its geometric
properties as well. It is shown that adding sufficiently large edges as well as cutting away sufficiently small
edges leads to a decrease of the
spectral gap.
Corresponding explicit criteria are given.
\end{abstract}

\section{Introduction}

Current paper is devoted to spectral properties of Laplace operators on metric graphs also known as {\it quantum graphs}
\cite{BeCaFuKu,ExSe,GePa,KoSch1,Ku04,Ku05}. Our studies were
inspired by classical results going back to M.\,Fiedler \cite{Fie} on the second eigenvalue of discrete graphs and by recent paper by P.\,Exner and M.\,Jex
on the ground state for quantum graphs with delta-coupling \cite{ExJe}. M.\,Fiedler proposed to call the second lowest (the first excited) eigenvalue
of the discrete Laplacian by {\it algebraic connectivity} of the corresponding discrete graph. Such name is explained by close relations between the
algebraic connectivity and standard vertex and edge connectivities. P.\,Exner and M.\,Jex investigated behavior of the ground state for
(continuous) Laplacians on metric graphs as one of the edges is shortened or extended. It appeared that the bound state may increase 
as the length of an edge is increasing, also the opposite behavior may be expected. 

Our goal is to study behavior of the first excited eigenvalue when either edges are delete or added to a metric graph. 
Bearing in mind that the ground state for standard Laplacians is zero, the first excited eigenvalue gives us
the spectral gap (of course, provided the graph is connected).
Spectral properties of
quantum graphs, especially with equilateral lengths of edges, are closely related to spectral properties of corresponding discrete Laplacians. Therefore
one might expect that qualitative behavior of eigenvalues for discrete and continuous Laplacians just coincide. It appears that the spectral gap
for discrete and continuous Laplacians may behave differently as edges are added or deleted without
changing the vertex set. This is connected with the fact that adding an edge
to a discrete graph does not change the phase space, while adding an edge to a metric graph enlarges the corresponding phase space.

Adding or deleting an edge without changing the vertex set changes graph's Euler characteristics. It has been proven that the Euler
characteristics is determined by spectral asymptotics and therefore cannot be recovered from the first few eigenvalues alone \cite{FuKuWi,Ku06,Ku08},
unless the metric graph has a basic length.

We would like to investigate the spectral gap in relation to graph's connectivity and geometry.  It has already been proven in \cite{Fr2005,KuNa} that the graph
formed by just one edge (or a chain of edges) has the lowest spectral gap among all quantum graphs having the same total length. Therefore
it is natural to expect that the spectral gap increases with connectivity. Adding an edge to a graph increases its connectivity, but the total length $ \mathcal L $
increases as well. Increase of the total length may lead to a decrease of the spectral gap, since in accordance to Weyl's law the eigenvalues satisfy
the asymptotics $ \lambda_n \sim \left(
\frac{\pi}{\mathcal L} \right)^2 n^2. $ Similarly, deleting an edge may lead to both decrease and increase of the spectral gap. We study these phenomena in more details
starting from addition of edges.

The paper is organized as follows. We prove first few elementary classical facts about spectral gap for discrete Laplacians. We continue then
with quantum graphs and study behavior of the spectral gap as two vertices are glued into one or as an edge is added between two already existing
vertices. After that we change our point of view and study the case where an edge is cut at a certain internal point or where a whole edge is deleted.
It appears that the spectral gap may grow even if a whole interval is cut away from the metric graph. Explicit estimates
for the length of the edge that can be cut away are obtained.

\section{Discrete graphs (warming up)}

Let  $ G $ be a discrete graph with $ M $ vertices and $ N $ edges connecting some of the vertices.
Then the corresponding Laplace operator $ L (G) $ is defined on
the finite dimensional space $ \ell_2 (G) =  \mathbb C^M $ by the following formula \cite{CvDoSa,CvRoSi,Mo}
\begin{equation} \label{ldisc}
\left(L(G) \psi \right) (m) = \sum_{n \sim m} \left( \psi(m) - \psi (n) \right),
\end{equation}
where the sum is taken over all neighboring vertices. The Laplace operator can also be defined using the
connectivity matrix $ C = \{ c_{nm} \}$
$$ c_{nm} = \left\{
\begin{array}{ll}
1, & \mbox{the vertices $n $ and $ m $ are neighbors}, \\
& \mbox{{\it i.e.} connected by an edge}; \\[2mm]
0, & \mbox{otherwise},
\end{array} \right. $$
and the valence matrix $ V = {\rm diag}\, \{ v_1, v_2, \dots, v_M \}, $ where $ v_m $ are the valencies (degrees) 
of the corresponding vertices
$$ L (G) = V - C, $$
which corresponds to the matrix realisation of the operator $ L(G) $ in the canonical basis given by
the vertices.
In the literature one may find another definitions for discrete Laplacians. In \cite{Chung} one uses the congruent matrix
\begin{equation}
\mathbb L (G) = V^{-1/2} L(G) V^{-1/2}  = I -  V^{-1/2} C V^{-1/2}.
\end{equation}
Such definition of the Laplacian matrix is consistent with the eigenvalues analysis in spectral geometry. Another one Laplacian
matrix connected with the averaging operator is similar to the previous one
\begin{equation}
\mathbf L (G) = V^{-1} L(G) = V^{-1/2} \mathbb L (G) V^{1/2}
\end{equation}
is important for studies of quantum graphs, since its eigenvalues are closely related to the spectrum of
corresponding equilateral graphs \cite{Cattaneo}.

In the current section we
discuss briefly spectral properties of the standard Laplace matrix $ L(G) $ given by \eqref{ldisc},  first of all in relation to
the set of edges.

Since the Laplace operator is uniquely defined by the discrete graph $ G $ its eigenvalues $ \lambda_0 \leq \lambda_1 \leq \dots \leq \lambda_{M-1} $ are usually referred to
as the eigenvalues of $ G. $
The ground state corresponding to $\lambda_0 = 0 $ has eigenfunction  $\psi_0 = 1$, where $ 1 \in \mathbb C^{M} $ denotes the vector built up of ones on $ G $.
The multiplicity of the ground state coincides with the number of connected components in $ G $. In order
to avoid artificial complications only connected graphs will be considered in the sequel.
Then the {\bf spectral gap} $ \lambda_1 - \lambda_0 $ for the discrete Laplacian coincides with 
the first excited eigenvalue $ \lambda_1 . $

 The spectral gap of a discrete graph is a monotonously increasing function of the set of edges. In other words,
 adding an edge always causes increase of the second eigenvalue or keeps it unchanged, provided we have the same set of vertices.
\begin{prop}\label{prop1}
Let $ G $ be a connected discrete graph and
 let $G'$ be a discrete graph obtained from $G$ by adding one edge between the vertices $m_1$ and $m_2$. Let $L$ denote the
  discrete Laplacian defined by \eqref{ldisc}. Then the following holds:
    \begin{enumerate}
        \item The first excited eigenvalues satisfy the inequality:
            \[\lambda_1(G)\leq \lambda_1(G').\]
        \item The equality $\lambda_1(G) = \lambda_1(G')$ holds if and only if the second eigenfunction $\psi_1^G$ on the  graph $G$ may be chosen attaining equal values at the vertices $m_1$ and $m_2$
                \[\psi_1^G(m_1) = \psi_1^G(m_2).\]
    \end{enumerate}
\end{prop}

\begin{proof}
    The first statement follows from the fact that
        \begin{equation}
            L(G') - L(G) = \left(
                \begin{array}{ccccc}
                     & \vdots & & \vdots &  \\
                     \ldots & 1 & \ldots & -1 & \ldots\\
                     & \vdots & & \vdots & \\
                     \ldots & -1 & \ldots & 1 & \ldots\\
                     & \vdots & & \vdots &
                \end{array}
              \right)
        \end{equation}
is a matrix with just four non-zero entries. It is easy to see that the matrix is positive semi-definite, since the eigenvalues are $0$ (with the multiplicity $M-1$) and $2$ (simple eigenvalue) and therefore $L(G')-L(G)\geq 0$ which implies the first statement.

To prove the last assertion let us recall that $\lambda_1(G')$ can be calculated using Rayleigh quotient
        \[
        \lambda_1(G') = \min_{\psi \bot 1} \frac{\langle\psi, L(G')\psi\rangle}{\langle \psi,\psi\rangle}\geq \min_{\psi \bot 1} \frac{\langle \psi, L(G) \psi \rangle}{\langle \psi, \psi \rangle}
        = \lambda_1 (G).
        \]
        Hence the trial function $ \psi $ should be chosen orthogonal to the ground state, {\it i.e.} having mean value zero.
We have equality in the last formula if and only if $\psi$ minimizing the first and the second quotients can be chosen such that
 $(L(G')- L(G))\psi = 0,$ i.\,e. $\psi(m_1) = \psi(m_2)$.

\end{proof}

Next we are interested in what happens if we add a pending edge, i.\,e. an edge connected to the graph at one already existing node.

\begin{prop}\label{prop2}
Let $G$ be a connected discrete graph and let $G'$ be another graph obtained from $G$ by adding one vertex and one edge between the new vertex and the vertex $m_1$. Then the following holds:

\begin{enumerate}
        \item{The first excited eigenvalues satisfy the following inequality:
            \[\lambda_1(G)\geq \lambda_1(G').\]}
        \item{The equality $\lambda_1(G) = \lambda_1(G')$ holds if and only if every eigenfunction $\psi_1^G$
        corresponding to $ \lambda_1 (G) $ is equal to zero at $m_1$
            \[\psi_1^G(m_1) = 0.\]}
\end{enumerate}
\end{prop}

\begin{proof}
    Let us define the following vector on $G'$:
        \[\varphi(n):=\left\{
            \begin{array}{ll}
                \psi_1^G(n), & \text{on } G,\\
                \psi_1^G(m_1)& \text{on } G'\backslash G.
            \end{array}\right.\]
    This vector is not orthogonal to the zero energy eigenfunction $ 1 \in \mathbb C^{M+1} $, where we keep the same notation $1$ 
    for the vector build up of ones now on $G'$. Therefore consider the nonzero vector $\gamma$ shifted by a constant $c$
        \[\gamma(n):=\varphi(n)+c.\]
   Here $c$ is chosen so that the orthogonality condition in  $l_2(G') = \mathbb C^{M+1} $ holds
        \[0 = \langle \gamma, 1\rangle_{l_2(G')} =  \underbrace{ \langle \psi_1^G, 1 \rangle_{l_2(G)} }_{=0} +\psi_1^G(m_1) + c M',\]
    where $M' = M+1$ is the number of vertices in $G'$. This implies
        \[c = -\frac{\psi_1^G(m_1)}{M'}.\]
    Using this vector the following estimate on the first eigenvalue may be obtained:
        \begin{eqnarray} \label{enq1}
            \lambda_1(G')\leq \frac{\langle L (G') \gamma,\gamma\rangle_{l_2(G')}}{\|\gamma\|^2_{l_2(G')}} = \frac{\langle L (G) \psi_1^G, \psi_1^G\rangle_{l_2(G)}}{\|\psi_1^G\|_{l_2(G)}^2+c^2 M+|\psi_1^G(m_1)+c|^2}\leq \lambda_1(G).
        \end{eqnarray}
    The last inequality follows from the fact that
        \[\langle L (G) \psi_1^G, \psi_1^G\rangle_{l_2(G)} = \lambda_1(G) \|\psi_1^G\|^2,\]
    and
        \[\|\psi_1^G\|_{l_2(G)}^2+c^2 M+|\psi_1^G(m_1)+c|^2 \geq \|\psi_1^G\|_{l_2(G)}^2.\]

    Note, that we have equality if and only if $c=0$ and $|\psi_1^G(m_1) + c|^2=0$ which implies $\psi_1^G(m_1)=0.$
    If there exists a $ \psi_1^G $, such that $ \psi_1^G (m_1) \neq 0 $, then the inequality in \eqref{enq1} is strict and we get
    $$ \lambda_1 (G) > \lambda_1 (G'). $$
\end{proof}

We see that the first excited eigenvalue has a tendency to decrease if a pending edge is attached to a graph. It is clear from the proof
that gluing instead of one edge any connected graph would lead to the same result, provided there is just one contact vertex. If the number
of contact vertices is larger, then the spectral gap may increase as shown in Proposition \ref{prop1}.

\section{Quantum graphs: definitions}

Let $ \Gamma $ be a metric graph formed by $ N = N(\Gamma) $ compact edges $ E_n = [x_{2n-1}, x_{2n} ]$, $ n= 1,2, \dots, N $
(identified with intervals on $ \mathbb R $) joined together at $ M = M(\Gamma) $ vertices (nodes). The free Laplace
operator $ L^{\rm st} (\Gamma) = - \frac{d^2}{dx^2} $ is defined in the Hilbert space $ L_2 (\Gamma) = \cup_{n=1}^N L_2 (E_n) $
on functions $ u \in \cup_{n=1}^N W_2^2 (E_n) $ satisfying standard matching conditions at the vertices
$ V_m, m= 1,2, \dots, M $
\begin{equation}
\left\{
\begin{array}{l}
\mbox{$u $ is continuous at $ V_m $},\\
\mbox{the sum of normal derivatives is zero}.
\end{array} \right.
\end{equation}
The free Laplacian is self-adjoint in $ L_2 (\Gamma) $ and is uniquely determined by the metric graph $ \Gamma.$
The quadratic form of $ L^{\rm st} $ is defined on the domain $ \stackrel{c}{W^1_2}(\Gamma) $ consisting of all functions from
$ \cup_{n=1}^N  W_2^1 (E_n) $ which are also continuous at the vertices.\footnote{Any function from the Sobolev space $ W_2^1 (\Gamma) $
        is continuous inside each edge, but such functions are not necessarily continuous at the vertices.}
Since the domain of the operator is invariant under complex conjugation, the corresponding eigenfunction may be
chosen real. Therefore in order to simply our presentation we assume that the eigenfunctions are real.

The spectrum of the Laplacian is
discrete $ \lambda_0 = 0 \leq \lambda_1 \leq \lambda_2 \leq \ldots $ and will be referred to as the spectrum of $ \Gamma. $ If $ \Gamma $ is connected,
then the ground state $ \lambda_0 = 0 $ has multiplicity one and the corresponding eigenfunction is $ \psi_0^\Gamma = 1. $
Since only connected graphs will be considered, the spectral gap $ \lambda_1 (\Gamma) - \lambda_0 (\Gamma) $ coincides
with the energy of the first excited state $ \lambda_1. $

\section{Increasing connectivity - gluing vertices together}

As was already mentioned, the spectral gap was extensively investigated for discrete graphs. Our goal here is to study
the spectral gap for Laplacians on metric graphs especially in relation to connectivity of the underlying metric graphs.
More precisely, our aim is to prove analogs of Proposition \ref{prop1} and \ref{prop2} for quantum graphs. Our original idea was to study behavior 
of the spectral gap when a new edge is added to the original metric graph. But this procedure increases the total length of the
graph and therefore it is not surprising that the spectral gap has tendency to decrease in contrast to Proposition \ref{prop1}
(see Theorem \ref{th:add} below). Therefore let us start our studies by presenting a direct analog of Proposition \ref{prop1}
for quantum graphs. The corresponding theorem answers the following question: what happens to the spectral gap
if two vertices in a metric graph are joined into one common vertex. This procedure does not change the set of edges and therefore the total length of
the graph is also preserved, but increases graph's connectivity instead.

\begin{theo} \label{th:join}
Let $ \Gamma $ be a connected metric graph and let $ \Gamma' $ be another metric graph obtained from $ \Gamma $ by joining
together two of its vertices, say $ V_1 $ and $ V_2. $ Then the following holds:
\begin{enumerate}
\item The spectral gap satisfies the inequality
\begin{equation} \label{ens}
\lambda_1 (\Gamma) \leq \lambda_1 (\Gamma').
\end{equation}
\item The equality $ \lambda_1 (\Gamma) = \lambda_1 (\Gamma') $ holds if and only if the eigenfunction $ \psi_1 $ corresponding to
the first excited state can be chosen such that
\begin{equation} \label{eqva} \psi_1 (V_1) = \psi_1 (V_2). 
\end{equation}
\end{enumerate}
\end{theo}
\begin{proof}
The first excited state can be calculated by minimizing the Rayleigh quotient $ \displaystyle \frac{ \| u' \|^2 }{ \|  u \|^2} $ corresponding
to the standard Laplacian
over the set of functions from the domain of the quadratic form which are in addition
orthogonal to the ground state eigenfunction $ \psi_0 = 1. $ For the original graph $\Gamma $ the domain of the quadratic form
consists of all $ W_2^1 (\Gamma) $ functions which are continuous at all vertices of $ \Gamma. $ The corresponding set for $ \Gamma' $
is characterized by one additional condition $ u (V_1) = u(V_2) $ - continuity of the function at the new vertex $ V_1 \cup V_2. $
Inequality \eqref{ens} for the corresponding minima follows.

To prove the second statement we first note that if the minimizing function $ \psi_1 $ for $ \Gamma $ satisfies in addition \eqref{eqva}, then
the same function is a minimizer for $ \Gamma'$ and the corresponding eigenvalues coincide. 
It is clear since the domain of the quadratic form keeps only the continuity
of functions at the vertices. Conversely
if $  \lambda_1 (\Gamma) = \lambda_1 (\Gamma') $,
then the eigenfunction for $ L^{\rm st}( \Gamma') $ is also a minimizer for the Rayleigh quotient for $ \Gamma $ and therefore is an eigenfunction
for $ L^{\rm st} (\Gamma)$ satisfying in addition \eqref{eqva}. 

\end{proof}

Proposition \ref{prop1} and Theorem \ref{th:join} appear to be rather similar at first glance. But the reasons for the spectral 
gap to increase
are different. In the case of discrete graphs the difference between the Laplace operators is a nonnegative matrix. For quantum graphs
the differential operators are identical, but inequality \eqref{ens} is valid due to the fact that the opposite
inequality holds for the domains of the quadratic forms.

\begin{corollary} Theorem \ref{th:join} implies that the flower graph consisting of $ N $ loops attached to one vertex has the largest
spectral gap among all graphs formed by a given set of edges.
\begin{figure}

\centerline{ \includegraphics[width=0.2\textwidth]{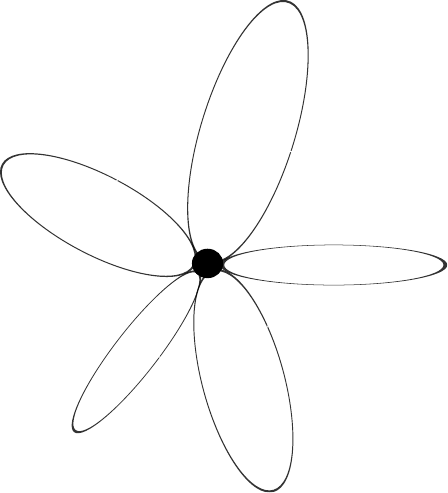} }
\caption{Flower graph.}

 \label{Fig2}

\end{figure}
\end{corollary}

\section{Adding an edge}

Our goal in this section is to study behavior of the spectral gap as an extra edge is added to the metric graph.
We start by proving a direct analog of Proposition \ref{prop2}.

\begin{theo} \label{Th2}
Let $\Gamma$ be a connected metric graph and let $\Gamma'$ be another graph obtained from $\Gamma$ by adding one vertex and one edge
 connecting the new vertex with the vertex $ V_1$. 
 \begin{enumerate}
 \item The first eigenvalues satisfy the following inequality:
            \[\lambda_1(\Gamma)\geq \lambda_1(\Gamma').\]
\item The equality $ \lambda_1 (\Gamma) = \lambda_1 (\Gamma') $ holds if and only if
every eigenfunction $ \psi_1^\Gamma $ corresponding to $ \lambda_1 (\Gamma) $ is equal to zero at $ V_1 $
$$ \psi_1^\Gamma (V_1) = 0. $$
\end{enumerate}
\end{theo}

\begin{proof}

    Let us define the following function on $\Gamma'$:
        \[f(x):=\left\{
            \begin{array}{ll}
                \psi_1(x), & x \in \Gamma,\\
                \psi_1(V_1)& x \in \Gamma'\backslash \Gamma.
            \end{array}\right.\]
    This function is in general not orthogonal to the zero energy eigenfunction $1 \in L_2(\Gamma')$. Therefore consider the nonzero function $g$ differed from $ f $
     by a constant
        \[g(x):=f(x)+c,\]
    where $c$ is chosen so that the orthogonality condition in $L_2(\Gamma')$ holds
        \[0 = \langle g(x), 1\rangle_{L_2(\Gamma')} = \underbrace{\langle \psi_1,1 \rangle_{L_2(\Gamma)}}_{=0}+\psi_1(V_1) \ell + c\mathcal L',\]
    where $ \ell $ and  $\mathcal L'$ are the length of the added edge and the total length of $\Gamma'$ respectively.
    This implies
        \[c = -\frac{\psi_1(V_1) \ell}{\mathcal L'}.\]
    Using this vector the following estimate for the first eigenvalue may be obtained:
        \begin{eqnarray*}
            \lambda_1(\Gamma')\leq \frac{\langle L^{st}g,g\rangle_{L_2(\Gamma')}}{\|g\|^2_{L_2(\Gamma')}} = \frac{\langle L^{st} \psi_1, \psi_1\rangle_{L_2(\Gamma)}}{\|\psi_1\|_{L_2(\Gamma)}^2+c^2 \mathcal L +|\psi_1(V_1)+c|^2 \ell}\leq \lambda_1(\Gamma).
        \end{eqnarray*}
        Here $ \mathcal L $ denotes the total length of the metric graph $ \Gamma. $
The last inequality follows from the fact that
        \[\langle L^{st} \psi_1, \psi_1\rangle_{L_2(\Gamma)} = \lambda_1(\Gamma) \|\psi_1\|^2,\]
    and
        \[\|\psi_1^\Gamma\|_{L_2(\Gamma)}^2+c^2 \mathcal L+|\psi_1(V_1)+c|^2 \ell \geq \|\psi_1\|_{L_2(\Gamma)}^2.\]

    Note that in the last expression the equality holds if and only if $c=0$ and $|\psi_1(V_1) + c|^2=0$ which implies $\psi_1(V_1)=0.$ This proves the second assertion.

\end{proof}

In the proof of the last theorem we did not really use that  $ \Gamma' \setminus \Gamma $ is an edge. It is straightforward
to generalize the theorem for the case where $  \Gamma' \setminus \Gamma $ is an arbitrary finite connected graph
joined to $ \Gamma $ at one vertex only.

We return now to our original goal and investigate behavior of the spectral gap when an edge between two
vertices is added to a metric graph.

\begin{theo}\label{th:add}
    Let $\Gamma$ be a connected metric graph and $L^{st}$ -- the corresponding free Laplace
    operator. Let $\Gamma'$ be a metric graph obtained from $\Gamma$ by adding an edge
     between the vertices $V_1$ and $V_2$.
 Assume that the eigenfunction $\psi_1$ corresponding to the first excited eigenvalue can be chosen such that
                \begin{equation}\label{extra}
                    \psi_1(V_1) = \psi_1(V_2).
                \end{equation}
            Then the following inequality for the second eigenvalues hold:
                \[\lambda_1(\Gamma) \geq \lambda_1(\Gamma').\]
\end{theo}

\begin{proof}
To prove the inequality let us consider the eigenfunction $\psi_1(\Gamma)$ for $L^{st}(\Gamma).$ We introduce a new function
on $ \Gamma' $
    \[f(x)=\left\{ \begin{array}{ll}
                            \psi_1(x), & x \in \Gamma,\\
                            \psi_1(V_1) \; \; (=\psi_1(V_2)) & x \in \Gamma'\backslash \Gamma.
                    \end{array}\right.\]
This function is not orthogonal to the constant function. Let us adjust the constant $c$ so that the nonzero function
$g(x) = f(x) + c$ is orthogonal to $ 1$ in 
$L_2(\Gamma')$:\footnote{In what follows we are going to use the same notation $ 1 $ for the functions identically equal to one on both
metric graphs $ \Gamma $ and $ \Gamma'$.}
    \[0 = \langle g(x),1  \rangle_{L_2(\Gamma')} = \underbrace{\langle \psi_1(x),1\rangle_{L_2(\Gamma)}}_{=0} + \psi_1(V_1)\ell + c \mathcal L'=0,\]
where $ \ell $ is the length of the added edge and $\mathcal L'$ is the total length of the graph $\Gamma'$, as before. 
We have used that the eigenfunction $ \psi_1 $ has mean value zero, {\it i.e.} is orthogonal to the ground state. This implies
    \[c= -\frac{\psi_1(V_1) \ell}{\mathcal L'}.\]
Now we are ready to get an estimate for $\lambda_1(\Gamma')$ using Rayleigh quotient
    \[\lambda_1(\Gamma') \leq \frac{\langle L^{st}(\Gamma')g,g\rangle_{L_2(\Gamma')}}{\|g\|^2_{L_2(\Gamma')}}.\]
The numerator and denominator can be evaluated as follows
   \begin{align*} \langle L^{st}(\Gamma')g,g\rangle_{L_2(\Gamma')} &= \langle L^{st}(\Gamma)\psi_1,\psi_1 \rangle_{L_2(\Gamma)} = \lambda_1(\Gamma)\|\psi_1\|^2_{L_2(\Gamma)},\\[3mm]
         \|g\|^2_{L_2(\Gamma')} &= \|\psi_1 +c \|^2_{L_2(\Gamma)} + | \psi_1(V_1) + c|^2 \ell = \\
         & = \|\psi_1\|^2_{L_2(\Gamma)} + c^2 \mathcal L + |\psi_1(V_1) + c|^2 \ell \geq \\
         & \geq \|\psi_1\|^2_{L_2(\Gamma)}
    \end{align*}
It follows, that
    \[\lambda_1(\Gamma)\geq \lambda_1(\Gamma').\]
\end{proof}

Let us illustrate the above theorem by couple of examples:

\begin{example} \label{Ex1}
Let $ \Gamma $ be the graph formed by one edge of length $ a $. The spectrum of $L^{\rm st} (\Gamma)$ is
\[\sigma(L^{\rm st} (\Gamma)) = \left\{\left(\frac{\pi}{a}\right)^2 n^2\right\}_{n=0}^\infty.\]
All eigenvalues have multiplicity one.

Consider the graph $\Gamma'$ obtained from $ \Gamma $ by adding an edge of length $ b, $ so
that $ \Gamma' $ is  formed by two intervals of lengths $a$ and $b$ connected in parallel.
The graph $ \Gamma'$  is equivalent to the circle of length $a+b$. The spectrum  is:
    \[\sigma(L^{\rm st}(\Gamma')) = \left\{\left(\frac{2\pi}{a+b}\right)^2 n^2\right\}_{n=0}^\infty,\]
where all the eigenvalues except for the ground state have double multiplicity.

Let us study the relation between the first eigenvalues:
$$ \lambda_1 (\Gamma) = \frac{\pi^2}{a^2}, \; \, \lambda_1 (\Gamma') = \frac{4 \pi^2}{(a+b)^2}. $$
Any relation between these values is possible:
        \begin{align*}
            b>a & \: \Rightarrow \: \lambda_1(\Gamma) >  \lambda_1(\Gamma') ,\\
            b<a & \: \Rightarrow \: \lambda_1(\Gamma) <  \lambda_1(\Gamma') .
        \end{align*}
Therefore the first eigenvalue is not in general a monotone decreasing  function of the set of edges. The spectral gap decreases
 only if certain additional conditions are satisfied.
\end{example}

\begin{example} \label{Ex2}
    Consider, in addition to graph $ \Gamma' $ discussed in Example \ref{Ex1} ,
    the graph $\Gamma''$ obtained from $ \Gamma'$ by adding another one edge of length $ c�$
    between the same two vertices.
    Hence $ \Gamma'' $ is formed by three parallel edges of lengths $a,b$ and $c$.
     The first eigenfunction for $L^{\rm st} (\Gamma')$ can always be chosen so that its values at the vertices are equal.
      Then, in accordance with Theorem \ref{th:add}, the first eigenvalue for $ \Gamma'' $ is less or equal to the first eigenvalue for $ \Gamma' $:
        \[\lambda_1(\Gamma'')\leq \lambda_1(\Gamma').\]
        This fact can easily be supported by explicit calculations.
\end{example}

\begin{figure}

 \hspace{15mm}
 \includegraphics[width=0.2\textwidth]{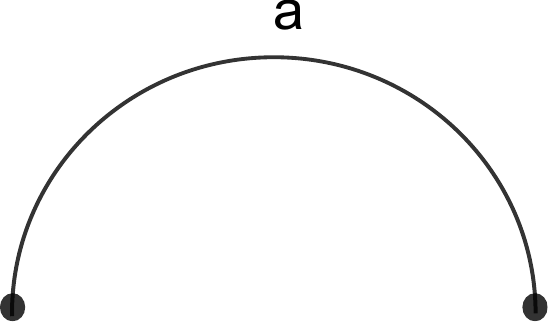} \vspace{-14mm}
 
 \hspace{50mm}
  \includegraphics[width=0.2\textwidth]{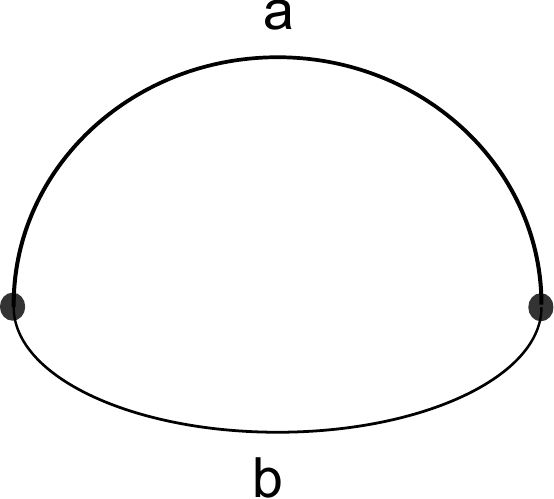}
   \hspace{5mm}
   \includegraphics[width=0.2\textwidth]{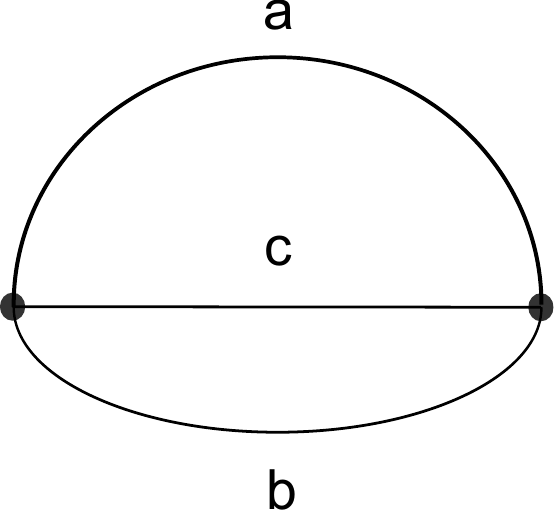}
\caption{Graphs $ \Gamma $, $ \Gamma'$, and $ \Gamma''.$}

 \label{Fig}

\end{figure}

Considered examples and proved theorems show that the spectral gap has a tendency to decrease, when a new 
sufficiently long edge is
added. It is not surprising, since addition of an edge increases  the total length of the graph, but the eigenvalues satisfy
Weyl's law and therefore are asymptotically close to $ (\pi n)^2/\mathcal L^2. $
This is in contrast to discrete graphs, for which addition of an edge does not lead to the increase of the number
of vertices.

Condition (\ref{extra}) in Theorem \ref{th:add} is not easy to check for non-trivial graphs and therefore it might be
interesting to obtain another explicit sufficient conditions.
In what follows we would like to discuss one such geometric condition ensuring that the spectral gap
drops as  a new edge is added to a graph. The main idea is to compare the length $ \ell $ of the new edge
with the total length of the original graph $ \mathcal L (\Gamma).$
It appears that if $ \ell > \mathcal L(\Gamma) $,   then
the spectral gap always decreases. We have already observed this phenomenon discussing Example \ref{Ex1}, where
behavior of $ \lambda_1 $ depended on the ratio between the lengths $ a $ and $ b. $ If $ b \equiv \ell > a \equiv \mathcal L (\Gamma) $, then the gap decreases.
It is surprising that the same explicit condition holds for arbitrary connected graphs $ \Gamma. $

\begin{theo} \label{th:ell}
Let $ \Gamma $ be a connected finite compact metric graph of length $ \mathcal L(\Gamma) $
and let $ \Gamma' $ be a graph constructed from $ \Gamma $ by
adding an edge of length $ \ell $ between certain two vertices. If
\begin{equation} \label{ass2}
 \ell > \mathcal L(\Gamma) ,
 \end{equation} then
the eigenvalues of the corresponding free Laplacians satisfy the estimate
\begin{equation} \label{ineq3}
      \lambda_1 (\Gamma) \geq \lambda_1 (\Gamma').
\end{equation}
\end{theo}
\begin{proof}
Let $ \psi_1$ be any eigenfunction corresponding to the first excited eigenvalue $ \lambda_1 (\Gamma) $
 of $L^{\text{st}}(\Gamma)$.  It follows that the minimum of the Rayleigh quotient is attained at
 $ \psi_1 $:
        \[\lambda_1(\Gamma) =\min_{u \in \stackrel{c}{W^1_2}(\Gamma): u \perp 1} \frac{\|u'\|^2_{L^2(\Gamma)}}{\|u\|^2_{L_2(\Gamma)}} = \frac{\|\psi_1'\|^2_{L_2(\Gamma)}}{\|\psi_1\|^2_{L_2(\Gamma)}},\]
        where $ \stackrel{c}{W^1_2}(\Gamma) $ denotes the set of continuous on graph $ \Gamma $ $ W_2^1$-functions.
 Let us denote by $ V_1 $ and $ V_2 $ the vertices in $ \Gamma$, where the new edge $ E $ of length $ \ell $ is attached.

The eigenvalue $\lambda_1(\Gamma')$ can again be estimated using Rayleigh quotient
 \begin{equation}\label{rayleigh}
        \lambda_1(\Gamma') =\min_{u \in \stackrel{c}{W^1_2}(\Gamma'): u \perp 1}
        \frac{\|u'\|^2_{L^2(\Gamma')}}{\|u\|^2_{L_2(\Gamma')}} \leq \frac{\|g'\|^2_{L_2(\Gamma)}}{\|g\|^2_{L_2(\Gamma)}},
\end{equation}
where $ g(x) $ is any function in  $ \stackrel{c}{W^1_2}(\Gamma') $ orthogonal to constant function in $ L_2 (\Gamma').$
Let us choose the trial function $ g $ of the form $ g(x) = f(x) + c $ where
 \begin{equation}
    f(x):=\left\{
    \begin{array}{ll}
    \psi_1(x), & x\in \Gamma,\\
    \gamma_1+\gamma_2\sin{\left(\frac{\pi x}{\ell}\right)} & x\in \Gamma'\backslash \Gamma = E = [-\ell/2, \ell/2],
    \end{array}\right.
\end{equation}
with $\gamma_1 = (\psi_1 (V_1) + \psi_1(V_2))/2$ and $\gamma_2 = (\psi_1(V_2)- \psi_1 (V_1))/2$. 
Here we assumed that left end point of the interval is connected to $ V_1 $ and the right end point - to $ V_2. $
The function $ f$ obviously
belongs to $ \stackrel{c}{W^1_2}(\Gamma') $, since it is continuous at $ V_1 $ and $V_2$ , but it is not necessarily orthogonal to the ground state eigenfunction
 $1$. The constant $c$ is adjusted in order to ensure the orthogonality
    \[\langle g,1 \rangle_{L_2(\Gamma')} = 0\]
    holds. The constant $ c $ can easily be calculated
    $$ 0 = \langle g, 1 \rangle_{L_2(\Gamma')} = c \mathcal L' +  \underbrace{\langle \psi_1, 1 \rangle_{L_2(\Gamma)}}_{=0}  + \int_{\ell/2}^{\ell/2}
    \left( \gamma_1 + \gamma_2 \sin \left(\frac{\pi x}{\ell}\right) \right) dx = c \mathcal L' + \gamma_1 \ell $$
\begin{equation} \label{eqc}
    \Rightarrow c = - \frac{\gamma_1 \ell}{\mathcal L'}.
\end{equation}

The function $ g $ can be used as a trial function in \eqref{rayleigh} to estimate the spectral gap.
Let us begin by computing the denominator using the fact that $ g $ is orthogonal to $ 1 $
    \begin{align}\label{denom}
    \| g \|^2_{L_2(\Gamma')} &= \|f  + c\|_{L_2(\Gamma')}^2 = \| f \|^2_{L_2(\Gamma')} -   \| c  \|^2_{L_2 (\Gamma')}  \nonumber \\
    &=  \|\psi_1\|_{L(\Gamma)}^2 + \int_{-\ell/2}^{\ell/2}\left(\gamma_1+\gamma_2\sin{\left(\frac{\pi x}{\ell}\right)}\right)^2\: dx - c^2 \mathcal L'  \nonumber \\
    &= \|\psi_1\|_{L(\Gamma)}^2 + \ell \gamma_1^2 + \frac{\ell}{2}\gamma_2^2- c^2 \mathcal L'
    \end{align}
The numerator yields
    \begin{align}\label{nom}
        \|g'\|^2_{L_2(\Gamma')} &= \|f'\|^2_{L_2(\Gamma')} = \|\psi_1'\|^2_{L_2(\Gamma)} + \int_{-\ell/2}^{\ell/2} \left(\gamma_2^2 \frac{\pi^2}{\ell^2}\cos^2\left(\frac{\pi x}{\ell}\right)\right)\: dx \nonumber\\
        &= \lambda_1(\Gamma)\|\psi_1\|^2_{L_2(\Gamma)} + \gamma_2^2 \frac{\pi^2}{2\ell}
    \end{align}
    After plugging \eqref{denom} and \eqref{nom} into \eqref{rayleigh} we obtain
    \[
    \lambda_1(\Gamma')\leq \frac{\lambda_1(\Gamma)\|\psi_1\|^2_{L_2(\Gamma)} + \gamma_2^2 \frac{\pi^2}{2\ell}}{\|\psi_1\|_{L(\Gamma)}^2 + \ell \gamma_1^2 + \frac{\ell}{2}\gamma_2^2- c^2 \mathcal L'}.
    \]
 Using \eqref{eqc}  the last estimate can be written as
 \begin{equation} \label{hjk}
    \lambda_1(\Gamma')\leq \frac{\lambda_1(\Gamma)\|\psi_1\|^2_{L_2(\Gamma)} + \gamma_2^2 \frac{\pi^2}{2\ell}}{\|\psi_1\|_{L(\Gamma)}^2 + \ell \gamma_1^2 \left(
    1- \frac{\ell}{\mathcal L'} \right) + \frac{\ell}{2}\gamma_2^2} \leq
    \frac{\lambda_1(\Gamma)\|\psi_1\|^2_{L_2(\Gamma)} + \gamma_2^2 \frac{\pi^2}{2\ell}}{\|\psi_1\|_{L(\Gamma)}^2 + \frac{\ell}{2}\gamma_2^2} ,
 \end{equation}
 where we used that $ \ell < \mathcal L' = \mathcal L + \ell. $
 It remains to take into account the following estimate for $ \lambda_1 $ proven in \cite{Fr2005,KuNa}
 \begin{equation} \label{estKuNa}
 \lambda_1 (\Gamma) \geq \left( \frac{\pi}{\mathcal L} \right)^2.
 \end{equation}
 Then taking into account \eqref{ass2} estimate \eqref{hjk} can be written as
 \begin{equation} \label{hjk2}
    \lambda_1(\Gamma')  \leq
    \frac{\lambda_1(\Gamma)\|\psi_1\|^2_{L_2(\Gamma)} + \lambda_1 (\Gamma) \gamma_2^2 \ell /2 }{\|\psi_1\|_{L(\Gamma)}^2 + \gamma_2^2 \ell/2 }  = \lambda_1 (\Gamma).
 \end{equation}
 The theorem is proven.
\end{proof}

Estimate \eqref{estKuNa} was crucial for our proof. It relates the spectral gap and the total length of the metric graph, {\it i.e.}
geometric and spectral properties of quantum graphs. It might be interesting to prove an analog of the last theorem for discrete
graphs. Proposition \ref{prop1} states that the spectral gap increases if one edge is added to a discrete graph. Adding a long
edge should correspond to adding a chain of edges to a discrete graph.

The previous theorem gives us a sufficient geometric condition for the spectral gap to decrease. Let us study now the case where
the spectral gap is increasing. Similarly, as we proved that adding one edge that is \emph{long enough} always makes the spectral
gap smaller (Theorem \ref{th:ell}), we claim that an edge that is \emph{short enough} makes it grow. We have already seen in Theorem \ref{th:join}
that adding an edge of zero length (joining two vertices into one) may lead to an increase of the spectral gap. It appears that criterium
for gap to decrease can be formulated explicitly in terms of the eigenfunction on the larger graph. Therefore let us change our point of view and
study behavior of the spectral gap as an edge is deleted.

\section{Decreasing connectivity - cutting edges} 

In the following section we are going to study spectral gap's behaviour when one of the edges is deleted. The result of
such procedure is not obvious, since cutting of an edge decreases  the total length of the metric graph and
one expects that the first excited eigenvalue increases. On the other hand cutting an edge decreases graph's
connectivity and therefore the spectral gap is expected to decrease. It is easy to construct examples when one
of these two tendencies prevails: Example \ref{Ex1} shows that the spectral gap may both decrease and increase
when an edge is deleted.

Let us discuss first what happens when one of the edges is cut in a certain internal point. Let $ \Gamma^* $ be a
connected
metric graph obtained from a metric graph $ \Gamma $ by cutting one of the edges, say $ E_1 = [x_1, x_2] $ 
at a point $ x^* \in (x_1, x_2). $ It will be convenient to denote by $ x_1^* $ and $ x_2^* $  the points on the two sides of the cut.         
 In other words, the graph $ \Gamma^* $ has precisely the same set of
edges and vertices as $ \Gamma $ except that the edge $ [x_1, x_2 ] $ is substituted by
two edges $ [x_1, x_1^*] $ and $ [ x_2^*, x_2 ] $ and two new vertices $ V_1^* = \{ x_1^*\} $ and $ V_2^* =  \{ x_2^* \}$ are added to the set of vertices.

The spectral gap for the graphs $ \Gamma $ and $ \Gamma^* $ can be calculated by minimising the same Rayleigh
quotient over the set of $ W_2^1$-functions with zero average. The only difference is that the
functions used to calculate $ \lambda_1 (\Gamma) $ are necessarily continuous at $ x^* $
$$ u (x_1^* ) = u (x_2^*) $$
(as functions from $ W_2^1 [x_1, x_2]$).  The functions used in calculating $ \lambda_1 (\Gamma^*) $
do not necessarily attain the same values at the points $ x_1^* $ and $ x_2^*. $ It follows that
$ \lambda_1 (\Gamma^*) \leq \lambda_1 (\Gamma), $ since the set of admissible functions is larger
for $ \Gamma^*. $ If the minimising function for $ \Gamma^* $ has the same values at $ x_1^* $ and
$ x_2^* $, then it is also an eigenfunction for $ L^{\rm st} (\Gamma) $ and therefore $ \lambda_1 (\Gamma^*)
= \lambda_1 (\Gamma). $ Moreover, if the spectral gap for the graphs is the same, then every function
minimising the quotient for $ \Gamma $ minimises the quotient for $ \Gamma^* $ as well and therefore
satisfies Neumann condition at $ x^*. $ It follows that every eigenfunction for $ L^{\rm st}(\Gamma) $
corresponding to $ \lambda_1 $ is also an eigenfunction for $ L^{\rm st}(\Gamma^*). $
The following theorem is proven.

\begin{theo} \label{Th:cut}
Let $ \Gamma $ be a connected metric graph and let $ \Gamma^* $ be another graph obtained
from $ \Gamma $ by cutting one of the edges at an internal point $ x^* $ producing two new
vertices $ V_1^* $ and $ V_2^*. $
\begin{enumerate}
\item The first excited eigenvalues satisfy the following inequality
\begin{equation}
\lambda_1 (\Gamma^*) \leq \lambda_1 (\Gamma).
\end{equation}
\item If $ \lambda_1 (\Gamma^*) = \lambda_1 (\Gamma) $ then
every eigenfunction of $ L^{\rm st}(\Gamma) $ corresponding to $ \lambda_1 (\Gamma) $ satisfies Neumann condition at the cut point $ x^* $: $ \psi'_1 (x^*) = 0. $   
If at least one of the eigenfunctions 
on $ \Gamma^* $ satisfies $ \psi_1^* (V_1^*) = \psi_1^* (V_2^*), $
then $ \lambda_1 (\Gamma^*) = \lambda_1 (\Gamma). $
\end{enumerate}
\end{theo}

This theorem is a certain reformulation of Theorem \ref{th:join} and implies that the spectral
gap has a tendency to decrease when an edge is cut in an internal point. Note that the total length of the graph is preserved this time.

\section{Deleting an edge}

Let us study now what happens if an edge is deleted, or if a whole interval is cut away from an edge
(without gluing the remaining intervals together). 
Let $ \Gamma $ be a connected metric graph as before and let $ \Gamma^* $ be a graph obtained from
$ \Gamma $ by deleting one of the edges. 

The following theorem proves a sufficient condition that guarantees that the spectral gap is decreasing as 
one of the edges is deleted.

\begin{theo}
    Let $\Gamma$ be a connected finite compact metric graph of the total length $\mathcal L$ and let  $\Gamma^*$ be another connected metric
    graph obtained from $\Gamma$ by deleting one edge of length $\ell$ between certain vertices $V_1$ and $V_2$.
    Assume in addition that
    \begin{equation} \label{assest}
    \left(
     \max_{\psi_1: L^{\rm st}  (\Gamma) \psi_1 = \lambda_1 \psi_1 } \frac{(\psi_1 (V_1) - \psi_1 (V_2))^2}{(\psi_1 (V_1) + \psi_1 (V_2))^2} \cot^2 \frac{k_1 \ell}{2} - 1 \right) \frac{k_1}{2} \cot \frac{k_1 \ell}{2} \geq (\mathcal L - \ell)^{-1},
    \end{equation}
    where $ \lambda_1 (\Gamma) = k_1^2, \; k_1 > 0 $ is the first excited eigenvalue of $ L^{\rm st} (\Gamma) $,
    then
    \begin{equation} \lambda_1 (\Gamma^*) \leq \lambda_1 (\Gamma). \end{equation}
\end{theo}

\begin{proof}
It will be convenient to denote the edge to be deleted by $ E = \Gamma \setminus \Gamma^*$ as well as to introduce notation
$ \mathcal L^* = \mathcal L-\ell $ for the total length of  $ \Gamma^*. $

    Let us consider any real eigenfunction $\psi_1$ on $ \Gamma $ corresponding to the eigenvalue $\lambda_1(\Gamma)$. We then define the function $g \in \stackrel{c}{W^1_2}(\Gamma^*)$ by
        \[
            g = \psi_1|_{\Gamma^*} + c,
        \]
    where the constant $c$ is to be adjusted so that $g$  has mean value zero on $ \Gamma^* $:
        \begin{equation}
        \label{orthog}
            \langle g,1\rangle_{L_2(\Gamma^*)} = 0.
        \end{equation}
 Straightforward calculations lead to
 $$
 0  =  \langle \psi_1, 1 \rangle_{L_2(\Gamma^*)} + c \mathcal L^*  =  - \langle \psi_1, 1 \rangle_{L_2(E)} + c \mathcal L^*
$$
 \begin{equation} \label{23}
 \Rightarrow c = \frac{\int_E \psi_1 (x) dx}{\mathcal L^* }.
 \end{equation}

    The function $ g $ can then be used to estimate the
     second eigenvalue $\lambda_1(\Gamma^*)$:
        \begin{equation}
        \label{est}
            \lambda_1(\Gamma^*) = \min_{u \in \stackrel{c}{W^1_2}(\Gamma^*): u\perp 1} \frac{\|u'\|^2_{L_2(\Gamma^*)}}{\|u\|^2_{L_2(\Gamma^*)}} \leq \frac{\|g'\|^2_{L_2(\Gamma^*)}}{\|g\|^2_{L_2(\Gamma^*)}}.
        \end{equation}
     Bearing in mind that $\langle \psi_1,1 \rangle_{L_2(\Gamma)} = 0$ and using \eqref{23} we evaluate the denominator in \eqref{est} first:
        \begin{align}
        \label{denom2}
            \|g\|^2_{L_2(\Gamma^*)} &= \|\psi_1 + c \|^2_{L_2(\Gamma^*)} = \int_{\Gamma}{(\psi_1 + c)^2}\: dx - \int_{E}{(\psi_1 + c)^2} \: dx = \nonumber \\
            & = \|\psi_1\|^2_{L_2(\Gamma)} - \int_E {\psi_1}^2 \: dx - \frac{1}{\mathcal L^*} \left(\int_E \psi_1 \: dx \right)^2.
        \end{align}
  The numerator similarly yields
        \begin{equation}
        \label{nom2}
            \|g'\|^2_{L_2(\Gamma^*)} = \int_{\Gamma}{(\psi_1')^2} \: dx - \int_{E} (\psi_1')^2 \: dx = \lambda_1(\Gamma)\|\psi_1\|^2_{L_2(\Gamma)} - \int_{E} (\psi_1')^2 \: dx.
        \end{equation}
    Plugging \eqref{denom2} and \eqref{nom2} into \eqref{est} we arrive at
        \begin{equation}
        \label{lambda1}
        \lambda_1(\Gamma^*) \leq \frac{\lambda_1(\Gamma)\|\psi_1\|^2_{L_2(\Gamma)} - \int_{E} (\psi_1')^2 \: dx}{\|\psi_1\|^2_{L_2(\Gamma)} - \int_E {\psi_1}^2 \: dx - \frac{1}{\mathcal L^*} \left(\int_E \psi_1 \: dx \right)^2}.
        \end{equation}

  Let us evaluate the integrals appearing in \eqref{lambda1} taking into account that $ \psi_1 $ is a solution to Helmholtz equation on the edge $ E $ which can be parameterized as $ E = [-\ell/2, \ell/2] $
  so that $ x = - \ell/2 $ belongs to $ V_1$ and $ x = \ell/2 $ to $ V_2 $
        \begin{equation} \label{efunc}
            \psi_1|_E(x) = \alpha \sin{(k_1 x)} + \beta \cos{(k_1 x)},
        \end{equation}
   where
        \begin{equation} \label{bval}  \alpha = - \frac{\psi_1(V_1)-\psi_1(V_2)}{2 \sin(k_1 \ell /2)}, \quad \beta = \frac{\psi_1(V_1)+\psi_1(V_2)}{2 \cos{(k_1 \ell/2)}}.\end{equation}
Direct calculations imply
$$
\begin{array}{lcl}
\displaystyle \int_E \psi_1(x) dx & = & \displaystyle \frac{2 \beta}{k_1} \sin \left(\frac{k_1 \ell}{2} \right); \\[2mm]
\displaystyle \int_E (\psi_1(x))^2 dx & = & \displaystyle \frac{\alpha^2 + \beta^2}{2} \ell - \frac{\alpha^2 - \beta^2}{2} \frac{\sin (k_1 \ell)}{k_1} ; \\[2mm]
\displaystyle \int_E (\psi_1' (x))^2 dx & = &  \displaystyle k_1^2 \left( \frac{\alpha^2 + \beta^2}{2} \ell + \frac{\alpha^2-\beta^2}{2} \frac{\sin ( k_1 \ell)}{k_1} \right).
\end{array} $$
        Inserting calculated values into \eqref{lambda1} we get
        \begin{equation}
        \lambda_1 (\Gamma^*) \leq \lambda_1 (\Gamma)
        \frac{\displaystyle \parallel \psi_1 \parallel^2_{L_2(\Gamma)} - \frac{\alpha^2+\beta^2}{2} \ell - \frac{\alpha^2 - \beta^2}{2} \frac{\sin (k_1 \ell)}{k_1}}
        {\displaystyle \parallel \psi_1 \parallel^2_{L_2(\Gamma)} - \frac{\alpha^2+\beta^2}{2} \ell + \frac{\alpha^2 - \beta^2}{2} \frac{\sin (k_1 \ell)}{k_1}
        - \frac{1}{\mathcal L^*} \frac{4 \beta^2}{\lambda_1 (\Gamma)} \sin^2 \left(\frac{k_1 \ell}{2} \right)}  .
        \end{equation}
     To guarantee that the quotient is not greater than $ 1 $ and therefore $ \lambda_1 (\Gamma^*) \leq \lambda_1 (\Gamma) $ it is enough that
     $$  \frac{\alpha^2 - \beta^2}{2} \frac{\sin (k_1 \ell)}{k_1} \geq - \frac{\alpha^2 - \beta^2}{2} \frac{\sin (k_1 \ell)}{k_1}
        + \frac{1}{\mathcal L^*} \frac{4 \beta^2}{\lambda_1 (\Gamma) } \sin^2 \left(\frac{k_1 \ell}{2} \right)$$
        \begin{equation} \label{star}
         \Longleftrightarrow \; \;
        \frac{k_1}{2} \left( \frac{\alpha^2}{\beta^2} -1 \right) \cot \left(\frac{k_1 \ell}{2}\right) \geq (\mathcal L^*)^{-1}.  
        \end{equation}
    Using \eqref{bval} the last inequality can be written as
    $$
         \left(
   \frac{(\psi_1 (V_1) - \psi_1 (V_2))^2}{(\psi_1 (V_1) + \psi_1 (V_2))^2} \cot^2 \left(\frac{k_1 \ell}{2}\right) - 1 \right) \frac{k_1}{2} \cot \left(\frac{k_1 \ell}{2}\right) \geq (\mathcal L^*)^{-1}.
     $$
     Remembering that the eigenfunction $ \psi_1 $ could be chosen arbitrary we arrive at \eqref{assest}. \end{proof}

Let us apply the above theorem to obtain an estimate for the length of the piece that can be cut 
from an edge so that the spectral gap still decreases. It appears that such estimate can be given in terms
of an eigenfunction $ \psi_1 $ corresponding to the first excited eigenvalue. 
Consider any edge in $ \Gamma $, say $ E_1 = [x_1, x_2] $ and choose an arbitrary
internal point $ x^* \in (x_1, x_2). $
Assume that we cut away an
interval of length $ \ell $ centred at $ x^*. $ Of course the
length $ \ell $ should satisfy the obvious geometric condition:  $ x_1 \leq x^*- \ell/2  $ and $ x^* + \ell/2 \leq x_2. $
We assume in addition that  
\begin{equation} \label{innn}
    \ell < \frac{\pi}{2 k_1}
    \end{equation}
    guaranteeing  in particular that the cotangent function in \eqref{assest} is positive.

The function $ \psi_1 $ on the edge $ E_1 $ can be written in a form
similar to \eqref{efunc}
\begin{equation}
\psi_1 (x) = \alpha \sin k_1 (x-x^*) + \beta \cos k_1 (x-x^*).
\end{equation}
Then formula \eqref{star} implies that the spectral gap decreases as the interval $ [x^*-\ell/2, x^* + \ell/2] $ is cut
away from the graph if 
\begin{equation} \label{adcond}
 \vert \alpha \vert > \vert \beta \vert. 
 \end{equation}
and the following estimate is satisfied
\begin{equation}
\cot \left( \frac{k_1 \ell}{2} \right) \geq \frac{2}{k_1 \mathcal L^* \left( \frac{\alpha^2}{\beta ^2} -1 \right)}.
\end{equation}
Using the fact that under condition \eqref{innn} we have $ \cot \left( \frac{k_1 \ell}{2} \right) \geq \frac{\pi}{2 k_1 \ell}$
the following explicit estimate on $ \ell $ can be obtained
\begin{equation} \label{estl}
\ell \leq \frac{\pi}{4} (\mathcal L - \ell) \left( \frac{\alpha^2}{\beta ^2} -1 \right) ,
\end{equation}
of course under condition \eqref{adcond}.
For the spectral gap not to increase it is enough that estimate \eqref{estl} is satisfied for at least one eigenfunction $ \psi_1 $:
\begin{equation}
\ell \leq \min \left\{ \frac{\pi}{2 k_1}, \; \;  \frac{\pi}{4} (\mathcal L-\ell) \max_{\psi_1: L^{\rm st}(\Gamma) \psi_1 = \lambda_1 \psi_1} \left( \frac{\alpha^2}{\beta^2} -1 \right) \right\},
\end{equation}
where we have taken into account \eqref{innn}.

We see that if the eigenfunction $ \psi_1 $ is sufficiently asymmetric with respect to the point $ x^* $
({\it i.e.} \eqref{adcond} is satisfied), then
a certain sufficiently small interval can be cut from the edge ensuring that the spectral gap decreases, despite
that the total length is also decreasing. Additional condition \eqref{adcond} was expected, since
if $ \psi_1 $ is symmetric with respect to $ x^* $, then the spectral gap may increase for any $ \ell. $
Really, one may imagine that deleting of the interval is performed in two stages. One cuts the edge $ E_1 $
at the point $ x^* $ first. Then one deletes the intervals $ [x^*-\ell/2, x^*_1] $ and $ [x^*_2, x^* + \ell/2] $.
If $ \alpha = 0 $ (symmetric function), then the spectral gap may be preserved in accordance to Theorem \ref{Th:cut}.
Deleting the pending edges (intervals $ [x^*-\ell/2, x^*_1] $ and $ [x^*_2, x^* + \ell/2] $)
may lead only to an increase of the spectral gap due to Theorem \ref{th:add}.

We have shown that deleting not so long edges or cutting away short intervals from edges may lead to
a decrease of the spectral gap despite the total length of the graph increases.
This effect reminds us of the phenomena
   discovered in \cite{ExJe}, where behavior of the spectral gap under extension of edges was discussed. It appeared that the ground state
   may decrease with the increase of the edge lengths, provided graphs are of complicated topology.

\end{document}